\documentclass[12pt]{amsart}

\usepackage{amsthm}

\usepackage{amssymb}
\usepackage{verbatim}
\usepackage[curve,matrix,arrow]{xy}
\usepackage{amsmath,amsfonts}
\usepackage{graphicx}
\usepackage{epsf}

\newcommand{\thmref}[1]{Theorem~\ref{#1}}

\newcommand{\eqnref}[1]{equation~(\ref{#1})}

\newcommand{\figref}[1]{Figure~\ref{#1}}

  {\end{list}}

\newtheorem{theorem}{Theorem}[section]

\newtheorem{corollary}[theorem]{Corollary}

\newcommand{\secref}[1]{\S\ref{#1}}

\def\mod{\text{mod}}

\begin{document}

\title[An extension of Lucas Theorem]
{An extension of Lucas Theorem}

\author{Zubeyir Cinkir\\
 \and \\
Aysegul Ozturkalan\\
}
\address{Zubeyir Cinkir\\
Department of Industrial Engineering\\
Abdullah Gul University\\
38100, Kayseri, TURKEY\\
}
\email{zubeyir.cinkir@agu.edu.tr}
\address{Aysegul Ozturkalan\\
Department of Engineering Sciences\\
Abdullah Gul University\\
38100, Kayseri, TURKEY\\}
\email{aysegul.ozturkalan@agu.edu.tr}


\begin{abstract}
We give elementary proofs of some congruence criteria to compute binomial coefficients in modulo a prime. These criteria are analogues to the symmetry property of binomial coefficients. We give extended version of Lucas Theorem by using those criteria. We give applications of these criteria by describing a method to derive identities and congruences involving sums of binomial coefficients.
\end{abstract}

\maketitle

\section{Introduction}\label{sec introduction}

It is because of the beauty and the richness of number theory, and mathematics in general, that you can come up with new formulations and discoveries even in the basic and fundamental subjects. 

Pascal's triangle and its properties have been investigated for a long time and it is a source of entertainment even for non-mathematicians. Its building blocks, binomial coefficients, are subject to numerous studies. For example, 
Lucas Theorem gives a remarkable method to compute binomial coefficients in modulo a prime. Namely, 
\begin{theorem}[Lucas Theorem, \cite{L2}]\label{thm Lucas}
For a given prime number $n$, let $m \geq n$ be an integer and let $k$ be an integer with $ 0 \leq k \leq m$.
Suppose $m=\sum_{i=0}^r m_i n^i$ and $k=\sum_{i=0}^r k_i n^i$ are expansions of $m$ and $k$ in base $n$ so that $r$ is an integer of appropriate size and $0 \leq m_i <n $ and $0 \leq k_i <n$ are integers for each $i=0, \, 1, \, \cdots, r$.
Then we have
\begin{equation*}
\binom{m}{k} \equiv \binom{m_r}{k_r}\binom{m-m_r n^r}{k-k_r n^r} \quad \mod \, \, \, n.
\end{equation*}
Moreover,
\begin{equation*}
\binom{m}{k} \equiv \binom{m_r}{k_r}\binom{m_{r-1}}{k_{r-1}} \cdots \binom{m_1}{k_1} \binom{m_0}{k_0} \quad \mod \, \, \, n.
\end{equation*}
\end{theorem}
Although, Lucas theorem is very powerful in reducing the computational complexity of $\binom{m}{k}$ in modulo $n$ when $n$ is a prime and $m \geq n$,
it is of no use when $m<n$ in which case $\binom{m}{k}$ can be very big comparing to $n$.
Moreover, the application of Lucas Theorem comes down to the computation of some other binomial coefficients of numbers that are smaller than $n$.  
For example, 
$$\binom{15683463}{10824637} \equiv \binom{248}{171} \binom{235}{205} \binom{230}{11}\quad \mod \, \, \, 251.$$
We still need to compute $\binom{248}{171}$, $\binom{235}{205}$ and $\binom{230}{11}$.

The symmetry property $\binom{m}{k}=\binom{m}{m-k}$ is also useful. Namely, it will be enough to compute the values of $\binom{m}{k}$ only for those $k$ in $\{ 0, \, 1, \, \cdots, \lfloor \frac{m}{2} \rfloor \}$. 

In this paper, we give elementary proofs of some other symmetry criteria so that the computations of binomial coefficients 
$\binom{m}{k}$ in modulo a prime number will be much more simpler. 
In \secref{sec main} (see \thmref{thm main1} and \thmref{thm main2}), we obtained the following equivalences, where $n$ is a prime, $k$ and $s$ are integers with $0 \leq s \leq k < n$:
\begin{equation}\label{eqn congruences}
\begin{split}
&(-1)^{k+s}\binom{n-1-s}{n-1-k}= (-1)^{k+s}\binom{n-1-s}{k-s} \equiv \binom{k}{s} \\
&= \binom{k}{k-s} \equiv (-1)^{s}\binom{n-1-k+s}{s}= (-1)^{s}\binom{n-1-k+s}{n-1-k} \quad \mod \, \, \, n.
\end{split}
\end{equation}

We can use \eqref{eqn congruences} to derive the following extended version of Lucas theorem (by using \eqref{eqn algo}):
\begin{theorem}[Extended Lucas Theorem]\label{thm Lucas ext}
For a given prime number $n$, let $m \geq n$ be an integer and let $k$ be an integer with $ 0 \leq k \leq m$.
Suppose $m=\sum_{i=0}^r m_i n^i$ and $k=\sum_{i=0}^r k_i n^i$ are expansions of $m$ and $k$ in base $n$ so that $r$ is an integer of appropriate size and $0 \leq m_i <n $ and $0 \leq k_i <n$ are integers for each $i=0, \, 1, \, \cdots, r$.
We set $t_i=\min \{ k_i, \, m_i-k_i \}$ for each $i=0, \, 1, \, \cdots, r$. 
Let us define  $\binom{m_i}{t_i}_n$ as follows
\begin{equation*}\label{eqn algorithm}
\begin{split}
\binom{m_i}{t_i}_n :=
\begin{cases} (-1)^{t_i} \binom{n-1-m_i+t_i}{n-1-m_i} , & \text{if $m_i+t_i > n- 1$}\\
(-1)^{t_i} \binom{n-1-m_i+t_i}{t_i}, & \text{if $m_i+t_i \leq n- 1$ and $2 m_i > n-1 +t_i$ }\\
\binom{m_i}{t_i}, & \text{otherwise 
}.
\end{cases}
\end{split}
\end{equation*}
Then we have
\begin{equation*}
\binom{m}{k} \equiv \binom{m_r}{t_r}_n \binom{m-m_r n^r}{k-k_r n^r} \quad \mod \, \, \, n.
\end{equation*}
Moreover,
\begin{equation*}
\binom{m}{k} \equiv \binom{m_r}{t_r}_n \binom{m_{r-1}}{t_{r-1}}_n \cdots \binom{m_1}{t_1}_n \binom{m_0}{t_0}_n \quad \mod \, \, \, n.
\end{equation*}
\end{theorem}

Observations given in \eqref{eqn congruences} are not only helpful for the computation of binomial coefficients in modulo a prime number, but also crucial to devise a new method of proving identities and congruences that involve summations of binomial coefficients.
In  \secref{sec identities}, we use \eqref{eqn congruences} either to give new proofs of a known identity or to derive new identities. Similarly, we use \eqref{eqn congruences} to establish (new) congruences in \secref{sec congruences}.

Having completed our paper, we noticed that symmetries in Pascal's triangle are studied and used in various articles such as \cite[Section 4]{GA}, \cite{JLXV}, \cite{K}, \cite{LC}, \cite{P}, \cite{R}, \cite{SB}. Almost all of them consider Pascal's triangle in modulo $n^k$ for a prime $n$, and they consider Pascal's triangle with $n^k$ or more rows. Many of them uses Lucas Theorem as a tool for their study. However, we focus only on the first $n-1$ rows, and work in modulo $n$. We use the symmetries to extend Lucas theorem and to devise a method of finding identities and congruences that involve Binomial coefficients.

\section{Congruences of Binomial Coefficients}\label{sec main}

In this section, we establish some congruence criteria that we call semi-symmetry properties of binomial coefficients in modulo a prime. These are given in \thmref{thm main1} and \thmref{thm main2}. We illustrate these criteria on Pascal's triangle. Then we describe an algorithm to apply these criteria to get a simplified equivalent form of a given binomial coefficients. This allows us to extend Lucas theorem. Then we give some examples.

The following simple equality will be very helpful.
For each integer $n$, $k$ and $s$ with $0 \leq s \leq k \leq n$, we have the cancellation identity \cite[pg 29]{BW}:
\begin{equation}\label{eqn product1}
\begin{split}
\binom{n}{k} \binom{k}{s}=\binom{n}{s}\binom{n-s}{k-s}.
\end{split}
\end{equation}
Let us recall the following well known fact.
\begin{theorem}\label{thm Leibniz2}
A positive integer $n$ is prime iff $\binom{n-1}{i} \equiv (-1)^i \quad \mod \, \, \, n$ for all $0 \leq i \leq n-1$.
\end{theorem}
In fact, only if part in \thmref{thm Leibniz2} was proved by Lucas \cite{L} in 1879.

Next, we give the first of our congruence criteria.
\begin{theorem}\label{thm main1}
Let $n$ be a prime. For each integers $k$ and $s$ with $0 \leq s \leq k < n$, we have
\begin{equation*}
\begin{split}
\binom{k}{s}    \equiv (-1)^{k+s} \binom{n-1-s}{k-s} =  (-1)^{k+s} \binom{n-1-s}{n-1-k}  \, \, \, \mod \, \, n.
\end{split}
\end{equation*}
\end{theorem}
\begin{proof}
By \eqnref{eqn product1}, we have
\begin{equation*}\label{eqn product1b}
\begin{split}
\binom{n-1}{k} \binom{k}{s}=\binom{n-1}{s}\binom{n-1-s}{k-s}.
\end{split}
\end{equation*}
Since $n$ is a prime, we have $\binom{n-1}{k} \equiv (-1)^k \, \, \, \mod \, \, n$, and similarly $\binom{n-1}{s} \equiv (-1)^s \, \, \, \mod \, \, n$. The result follows from these congruences and the equality. 
\end{proof}
We noticed that the congruence given in \thmref{thm main1} is obtained via different method in \cite[Corollary 7]{R}.
The following result was proved in \cite[Lemma 4]{P} by induction.
\begin{corollary}\label{cor main1}
Let $n$ be a prime. For $1 \leq i \leq n$ and $0 \leq k \leq n-i$ the following congruence holds:
\begin{equation*}\label{eqn cor1}
\begin{split}
\binom{n-i}{k} \equiv (-1)^k \binom{i-1+k}{i-1} \, \, \, \mod \, \, n.
\end{split}
\end{equation*}
\end{corollary}
\begin{proof}
With the given conditions on $i$ and $k$, we have $i-1+k \leq n-1$ and $0 \leq i-1 < n.$ Thus, we apply
\thmref{thm main1} to obtain
\begin{equation*}\label{eqn cor2}
\begin{split}
\binom{i-1+k}{i-1} \equiv (-1)^{i-1+k+i-1} \binom{n-1-(i-1)}{i-1+k-(i-1)} =(-1)^k \binom{n-i}{k} \, \, \, \mod \, \, n.
\end{split}
\end{equation*}
This completes the proof.
\end{proof}
Another result with the same flavor can be given as follows: 
\begin{theorem}\label{thm main2}
Let $n$ be a prime. For each integers $k$ and $s$ with $0 \leq s \leq k < n$, we have
\begin{equation*}
\begin{split}
\binom{k}{s} \equiv  (-1)^{s} \binom{n-1-k+s}{s} = (-1)^{s} \binom{n-1-k+s}{n-1-k} \, \, \, \mod \, \, n.
\end{split}
\end{equation*}
\end{theorem}
\begin{proof}
By the symmetry property, $\binom{k}{s}= \binom{k}{k-s}$.
Then by \thmref{thm main1}, 
$$\binom{k}{k-s} \equiv (-1)^{2k-s}\binom{n-1-k+s}{s} \equiv (-1)^{s}\binom{n-1-k+s}{s}  \, \, \, \mod \, \, n.$$ 
This gives what we want, since the equality in the theorem follows from the symmetry.
\end{proof}

Next we give some examples. By \thmref{thm main1} and the symmetry, we have
$$\binom{248}{171}=\binom{248}{77}=\binom{250-2}{79-2} \equiv (-1)^{81} \binom{79}{2} \equiv -3081 \equiv 182 \, \, \, \mod \, \, 251.$$ 
Similarly,
$$\binom{235}{205}= \binom{235}{30}= \binom{250-15}{45-15} \equiv (-1)^{60} \binom{45}{15}=344867425584 \equiv 27  \quad \mod \, \, \, 251,$$
and
$$\binom{230}{11}= \binom{250-20}{31-20} \equiv (-1)^{51} \binom{31}{20}=-\binom{31}{11}=-84672315 \equiv 25  \quad \mod \, \, \, 251.$$

We can use Lucas Theorem along with our extensions: 
\begin{equation*}\label{}
\begin{split}
\binom{15683463}{10824637} & \equiv \binom{248}{171} \binom{235}{205} \binom{230}{11}
\equiv \binom{79}{2} \binom{45}{15} \binom{31}{11} \\
& \equiv 182 \cdot 27 \cdot 25 = 122850 \equiv 111 \quad \mod \, \, \, 251.
\end{split}
\end{equation*}

%

\begin{figure}
\centering
\includegraphics[scale=0.86]{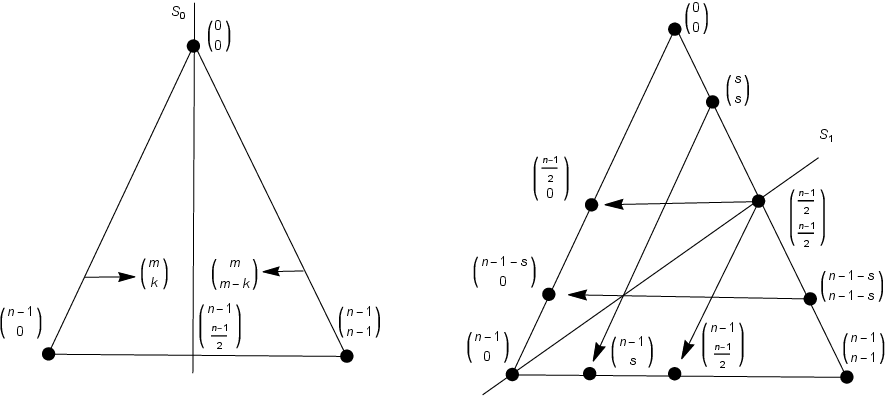} \caption{} \label{fig s0s1}
\end{figure}

The figure on the left in \figref{fig s0s1} shows the Symmetry Property $\binom{m}{k}=\binom{m}{m-k}$ and the symmetry line $S_0$ in the Pascal's Triangle, and the figure on the right in \figref{fig s0s1} shows the Semi-Symmetry Property in \thmref{thm main1}, $\binom{k}{s} \equiv (-1)^{k+s} \binom{n-1-s}{n-1-k}$ and the symmetry line $S_1$ in the Pascal's Triangle in modulo a prime $n$.



\begin{figure}
\centering
\includegraphics[scale=0.86]{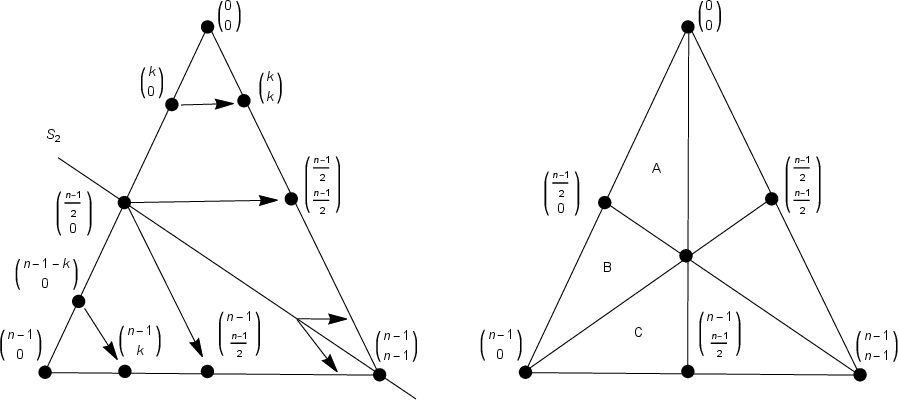} \caption{} \label{fig s2s3}
\end{figure}

The figure on the left in \figref{fig s2s3} shows the Semi-Symmetry Property described by \thmref{thm main2}, $\binom{k}{s} \equiv (-1)^s \binom{n-1-k+s}{s}$ and the symmetry line $S_2$ in the Pascal's Triangle in modulo a prime $n$.

The figure on the right in \figref{fig s2s3} shows the symmetry lines and some certain regions of Pascal's triangle. These regions are critical when we apply symmetry and semi-symmetry properties. In this figure,

Binomial coefficients $\binom{m}{k}$ on the symmetry line $S_0$, $S_1$ and $S_2$ are given by the conditions $k=m-k$,
$m+k=n-1$ and $2m=n-1+k$, respectively.
Therefore, $\binom{m}{k}$ is in the region $A$ if $k \leq m-k$ and $2m \leq n-1+k$. Likewise, 
$\binom{m}{k}$ is in the region of $B$ if $k \leq m-k$, $m + k \leq n-1$ and  $2m \geq n-1+k$.
Moreover, $\binom{m}{k}$ is in the region of $C$ if $k \leq m-k$, $m + k \geq n-1$.

\textbf{Question:} Given a binomial coefficient $\binom{m}{k}$ and a prime $n$ with $0 \leq k \leq m < n$, how can we apply symmetry and semi-symmetry properties effectively to get a simpler equivalent form of $\binom{m}{k}$ in modulo $n$? 

One can do this in various ways. Here is one possible method:

\textbf{Step 1:} If $k > m-k$, apply the symmetry property, i.e., replace $\binom{m}{k}$ by $\binom{m}{m-k}$.

\textbf{Step 2:} We are given $\binom{m}{k}$ with $k \leq m-k$ as step 1 is applied. If $m + k > n-1$, apply the semi-symmetry property given in \thmref{thm main1}, i.e., replace $\binom{m}{k}$ by $(-1)^{m+k} \binom{n-1-k}{n-1-m}$.

That is, if a binomial coefficient lies in the region $C$ of the Pascal's triangle in the second figure of \figref{fig s2s3}, applying step 2 moves it to the region $A \cup B$. Note that the conditions $k \leq m-k$ and $m + k > n-1$ ensures that $n-1-m < (n-1-k)- (n-1-m)$.

\textbf{Step 3:} We are given $\binom{m}{k}$ with $k \leq m-k$ and $m + k \leq n-1$ as step 1 and step 2 are applied.
If $2 m > n-1+k$, apply the semi-symmetry property given in \thmref{thm main2}, i.e., replace $\binom{m}{k}$ by $(-1)^{k} \binom{n-1-m+k}{k}$.

That is, if a binomial coefficient lies in the region $B$ of the Pascal's triangle in the second figure of \figref{fig s2s3}, applying step 3 moves it to the region $A$. Hence, the region $A$ is the fundamental region in Pascal's triangle modulo a prime.

Note that if step 2 is applied, then step 3 must be applied afterwards. To see this, one can check that output of step 2 (under the conditions of step 2) satisfies conditions of step 3.
If step 2 is not applied, we may need to apply step 3 or not. This will lead to the following shorter 
answer to the question above:

Given $\binom{m}{k}$, set $t:=\min \{ k, \, m-k \}$. Then we have the following congruence in modulo $n$:
\begin{equation}\label{eqn algo}
\begin{split}
\binom{m}{k} \equiv
\begin{cases} (-1)^t \binom{n-1-m+t}{n-1-m} , & \text{if $m+t > n- 1$}\\
(-1)^t \binom{n-1-m+t}{t}, & \text{if $m+t \leq n- 1$ and $2 m > n-1 +t$ }\\
\binom{m}{t}, & \text{otherwise (i.e. $m+t \leq n- 1$ and $2 m \leq n-1 +t$ ) }.
\end{cases}
\end{split}
\end{equation}
Next, we give some examples:

Compute $\binom{33}{20}$ in modulo $37$.

By step 1, $\binom{33}{20} = \binom{33}{13}$. Then

\begin{equation*}\label{eqn example2}
\begin{split}
\binom{33}{13}&\equiv (-1)^{33+13} \binom{23}{3} \, \, \, \mod \, \, \, 37, \quad \text{by step 2 as $13 < 33-13$ and $33+13 > 37-1$}.\\
&\equiv (-1)^{3}\binom{16}{3}  \equiv 32 \, \, \, \mod \, \, \, 37, \quad \text{by step 3 as $2  \cdot 23 \geq 37-1+3$}.
\end{split}
\end{equation*}

Compute $\binom{87}{40}$ in modulo $101$.
\begin{equation*}\label{eqn example3}
\begin{split}
\binom{87}{40} & \equiv (-1)^{87+40} \binom{60}{13} \, \, \, \mod \, \, \, 101, \quad \text{by step 2 as $40 < 87-40$ and $87+40 > 101-1$}.\\
&\equiv -(-1)^{13}\binom{53}{13}  \equiv  40 \, \, \, \mod \, \, \, 101, \quad \text{by step 3 as $2  \cdot 60 \geq 101-1+13$}.
\end{split}
\end{equation*}

\begin{figure}
\centering
\includegraphics[scale=0.85]{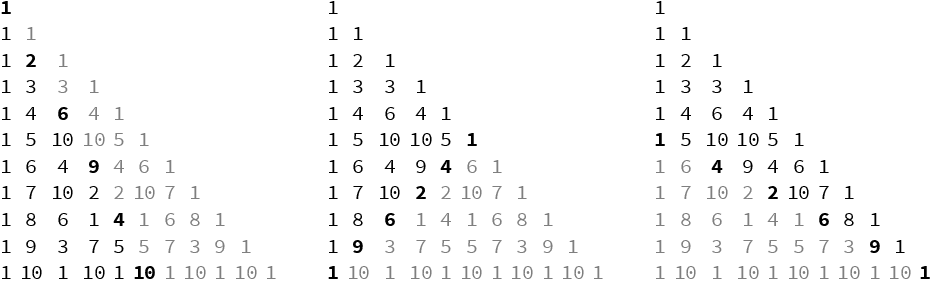} \caption{} \label{fig pt}
\end{figure}

When $n=11$, the symmetry line $S_0$ and the semi-symmetry lines $S_1$, $S_2$ are illustrated in \figref{fig pt}. Furthermore, applications of above steps in this case are illustrated in \figref{fig pt2}. 

\begin{figure}
\centering
\includegraphics[scale=0.87]{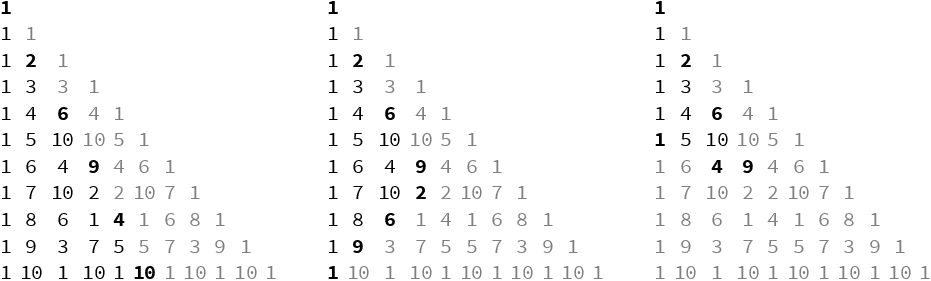} \caption{} \label{fig pt2}
\end{figure}

We used \cite{MMA} to prepare \figref{fig pt} and \figref{fig pt2}.

The rest of the paper is about applications of \thmref{thm main1} and \thmref{thm main2}.
The identities shown in \secref{sec identities} and the congruences given in \secref{sec congruences} are for the illustration purposes only. One can apply our method to various other cases.

\section{Obtaining identities}\label{sec identities}
In this section, we illustrate a method as an application of \thmref{thm main1} and \thmref{thm main2}.
It is about proving the summation identities involving binomial coefficients. Given such an identity, we apply any of these two theorems with a prime $p$ to rewrite the identity. This gives a congruence in modulo $p$. If we can choose the prime in such a way that the terms in the congruence are independent of $p$, the congruence is in fact an identity. Of course, the identity derived can be a known identity or a new identity. In any case, this gives a new effective method. Actually, this method shows that there is a duality between such identities.

The first equality in the following theorem was known (set $x=n$ in the equation $3.50$ in \cite{G}). Here, we derive it with a new proof.
\begin{theorem}\label{thm Vandermonde like1}
Let $n$ be a nonnegative integer. Then we have
\begin{equation*}\label{}
\begin{split}
\sum_{j=0}^n (-1)^j\binom{n}{j}\binom{2n-j}{n}=1, \quad \sum_{j=0}^n (-1)^j\binom{2n-j}{j}\binom{2n-2j}{n-j}=1.
\end{split}
\end{equation*}
\end{theorem}
\begin{proof}
By Vandermonde identity, for any $m > n$ we have
$\binom{m-1}{n}=\sum_{j=0}^n \binom{n}{j}\binom{m-1-n}{n-j}$. When $m$ is a prime, we have $\binom{m-1}{n} \equiv (-1)^n \, \, \, \mod \, \, m$. On the other hand, we use \thmref{thm main1} to derive $(-1)^{m+j}\binom{m-1-n}{n-j} \equiv (-1)^{n+1}\binom{2n-j}{n} \, \, \, \mod \, \, m$ if $m$ is a prime. Then we obtain
\begin{equation}\label{eqn Vid1}
\begin{split}
(-1)^n \equiv \sum_{j=0}^n (-1)^{m+n+j+1}\binom{n}{j}\binom{2n-j}{n}  \, \, \, \mod \, \, m.
\end{split}
\end{equation}
We can rewrite the congruence (\ref{eqn Vid1}) as follows:
\begin{equation}\label{eqn Vid2}
\begin{split}
1 \equiv \sum_{j=0}^n (-1)^{j}\binom{n}{j}\binom{2n-j}{n} \, \, \, \mod \, \, m.
\end{split}
\end{equation}
Note that the congruence (\ref{eqn Vid2}) is independent of $m$. By choosing $m$ large enough, we see that the congruence (\ref{eqn Vid2}) is in fact an equality. This gives the first equality in the theorem.

The second equality in the theorem follows from the first one if we use the cancellation identity 
$\binom{2n-j}{n}\binom{n}{j} = \binom{2n-j}{j}\binom{2n-2j}{n-j}$.
\end{proof}

\begin{theorem}\label{thm Vandermonde like2}
Let $s$ and $m$ be any positive integers. Then we have
\begin{equation*}\label{}
\begin{split}
\sum_{j=0}^s (-1)^j\binom{m+s+j}{s}\binom{s}{j}=(-1)^s, \quad \sum_{j=0}^s (-1)^j\binom{m+s+j}{j}\binom{m+s}{s-j}=(-1)^s.
\end{split}
\end{equation*}
\end{theorem}
\begin{proof}
We use the following equality \cite[pg 169]{GKP} (consider first equality at pg. 169 and set $n=0$) for any integers $m$, $s$ and $l$:
\begin{equation*}\label{}
\begin{split}
\sum_{j=0}^s \binom{l}{m+j}\binom{s}{j}=\binom{l+s}{l-m}.
\end{split}
\end{equation*}
Next, for a prime $p$, we take $l=p-1-s$ and rewrite this equation to obtain:
\begin{equation}\label{eqn V2}
\begin{split}
\sum_{j=0}^s \binom{p-1-s}{m+j}\binom{s}{j}=\binom{p-1}{p-1-s-m}.
\end{split}
\end{equation}
On the other hand, we have $\binom{p-1}{p-1-s-m} \equiv (-1)^{p-1-s-m}=(-1)^{s+m} \, \, \, \mod \, \, p$ by \thmref{thm Leibniz2}  and $\binom{p-1-s}{m+j}=\binom{p-1-s}{m+j+s-s} \equiv (-1)^{m+j+2s}\binom{m+s+j}{s} \, \, \, \mod \, \, p$ by \thmref{thm main1}. We use these in \eqref{eqn V2} to obtain:
\begin{equation}\label{eqn V3}
\begin{split}
\sum_{j=0}^s (-1)^j \binom{m+s+j}{s}\binom{s}{j} \equiv (-1)^s \, \, \, \mod \, \, p.
\end{split}
\end{equation}
Note that the congruence (\ref{eqn V3}) is independent of $p$. By choosing $p$ large enough, we see that the congruence (\ref{eqn V3}) is in fact an equality. This gives the first equality in the theorem.

The second equality in the theorem follows from the first one if we use the cancellation identity
$\binom{m+s+j}{s}\binom{s}{j} = \binom{m+s+j}{j}\binom{m+s}{s-j}$.
\end{proof}
Note that the identity given in \thmref{thm Vandermonde like2} is closely related to the identity with number $3.47$ in \cite{G} (consider the case $r=n$ in $3.47$).

\begin{theorem}\label{thm ratio1}
Let $s$ and $m$ be any positive integers. Then we have
\begin{equation*}\label{}
\begin{split}
\sum_{j=0}^m (-1)^j\frac{\binom{m}{j}}{\binom{s+j}{s}}=\frac{s}{s+m}.
\end{split}
\end{equation*}
\end{theorem}
\begin{proof}
We use the following equality \cite[pg 173]{GKP} for any integers $n \geq m \geq 0$:
\begin{equation*}\label{}
\begin{split}
\sum_{j=0}^m \frac{\binom{m}{j}}{\binom{n}{j}}=\frac{n+1}{n+1-m}.
\end{split}
\end{equation*}
Given positive integers $m$ and $s$, take a prime $p$ such that $p-1-s \geq m$. Next, we rewrite the equation above for $m$ and $n=p-1-s$:
\begin{equation}\label{eqn r1}
\begin{split}
\sum_{j=0}^m \frac{\binom{m}{j}}{\binom{p-1-s}{j}}=\frac{p-s}{p-s-m}.
\end{split}
\end{equation}
On the other hand, we have $\binom{p-1-s}{j}=\binom{p-1-s}{j+s-s} \equiv (-1)^{j+2s}\binom{s+j}{s} \, \, \, \mod \, \, p$ by \thmref{thm main1}. We use these in \eqref{eqn r1} to obtain:
\begin{equation}\label{eqn r2}
\begin{split}
\sum_{j=0}^m (-1)^j\frac{\binom{m}{j}}{\binom{s+j}{s}} \equiv \frac{s}{s+m} \, \, \, \mod \, \, p.
\end{split}
\end{equation}
Note that the congruence (\ref{eqn r2}) is independent of $p$. By choosing $p$ large enough, we see that the congruence (\ref{eqn r2}) is in fact an equality. This gives what we want.
\end{proof}
Note that the identity given in \thmref{thm ratio1} was listed as the identity with number $1.42$ and $4.30$ in \cite{G}.

\begin{theorem}\label{thm productsum1}
Let $s$ be any positive integer. Then we have
\begin{equation*}\label{}
\begin{split}
\sum_{j=0}^s (-1)^j j \binom{s+j}{s}\binom{s}{j}=(-1)^s (s+1) s .
\end{split}
\end{equation*}
\end{theorem}
\begin{proof}
We use the following equality \cite[pg 181]{GKP} for any integers $s$ and $n\geq 0$:
\begin{equation*}\label{}
\begin{split}
\sum_{j=0} j \binom{n}{j}\binom{s}{j}=s\binom{n+s-1}{n-1}.
\end{split}
\end{equation*}
Given a positive integer $s$, take an odd prime $p$ such that $p-1-s \geq 0$. Next, we rewrite the equation above for $s$ and $n=p-1-s$:
\begin{equation}\label{eqn s1}
\begin{split}
\sum_{j=0}^s j \binom{p-1-s}{j} \binom{s}{j}=s\binom{p-2}{s}.
\end{split}
\end{equation}
On the other hand, we have $\binom{p-2}{s} \equiv (-1)^s\binom{s+1}{1} \, \, \, \mod \, \, p$ and that $\binom{p-1-s}{j}=\binom{p-1-s}{j+s-s} \equiv (-1)^{j+2s}\binom{s+j}{s} \, \, \, \mod \, \, p$ by \thmref{thm main1}. We use these in \eqref{eqn s1} to obtain:
\begin{equation}\label{eqn s2}
\begin{split}
\sum_{j=0}^s (-1)^j j \binom{s+j}{j} \binom{s}{j} \equiv (-1)^s (s+1)s \, \, \, \mod \, \, p.
\end{split}
\end{equation}
Note that the congruence (\ref{eqn s2}) is independent of $p$. By choosing $p$ large enough, we see that the congruence (\ref{eqn s2}) is in fact an equality. This gives what we want.
\end{proof}

\begin{theorem}\label{thm sum new1}
Let $s$ and $n$ be positive integers. Then we have
\begin{equation*}\label{}
\begin{split}
\sum_{j=1}^n (-1)^j\binom{s-1}{j-1}\frac{1}{j}=(-1)^n\frac{1}{s} \binom{s-1}{n}-\frac{1}{s}.
\end{split}
\end{equation*}
\end{theorem}
\begin{proof}
We use the identity \cite[Identity 8.13 at pg 48]{G2}
$$\sum_{j=1}^n \binom{j+r-1}{r}\frac{1}{j}=\frac{1}{r} \binom{n+r}{r}-\frac{1}{r}. $$
For any given positive integers $n$ and $s$ with $s>n$, take $r=p-s$ for some prime $p > s$. Then we notice that $\binom{s-1}{j-1} \equiv (-1)^{j-1} \binom{p-1-s+j}{j-1}\, \, \, \mod \, \, p $ by \thmref{thm main2}. Therefore, 
\begin{equation}\label{eqn sn1}
\begin{split}
\sum_{j=1}^n (-1)^j \binom{s-1}{j-1} \frac{s}{j} & \equiv \sum_{j=1}^n \binom{p-1-s+j}{j-1} \frac{p-s}{j} \, \, \, \mod \, \, p \\
&=\sum_{j=1}^n \binom{p-1-s+j}{p-s} \frac{p-s}{j} \\
&\equiv   \binom{n+p-s}{p-s}-1 \, \, \, \mod \, \, p.
\end{split}
\end{equation}
For the right hand side, $\binom{n+p-s}{p-s}=\binom{p-1-(s-1)+n}{n} \equiv (-1)^n \binom{s-1}{n}\, \, \, \mod \, \, p$ by \thmref{thm main2} when $s-1 \geq n$. Thus, we can rewrite  (\ref{eqn sn1}) as follows:
\begin{equation}\label{eqn sn2}
\begin{split}
\sum_{j=1}^n (-1)^j \binom{s-1}{j-1} \frac{s}{j} \equiv (-1)^n \binom{s-1}{n}-1 \, \, \, \mod \, \, p.
\end{split}
\end{equation}
As the congruence (\ref{eqn sn2}) is independent of $p$, it must be an equality. Hence, the theorem is proved when $s > n$. 

If $s \leq n$, we only need to show that $\sum_{j=1}^s (-1)^j\binom{s-1}{j-1}\frac{s}{j}=-1$. Using $\binom{s}{j}=\frac{s}{j}\binom{s-1}{j-1}$ when $j \neq 0$, we need $\sum_{j=1}^s (-1)^j\binom{s}{j}=-1$ which clearly holds.
\end{proof}

\section{Obtaining congruences}\label{sec congruences}
In this section, we give examples of obtaining congruences in modulo a prime by applying \thmref{thm main1} and \thmref{thm main2}. This is interesting, because it shows the connections between seemingly unrelated congruences. 

\begin{theorem}\label{thm sum cong1}
Let $n$ be a prime. For any nonnegative integer $s < n-1$ we have
\begin{equation*}\label{}
\begin{split}
\sum_{k=1}^{n-1-s} \frac{1}{k}\binom{k+s}{s} \equiv -\sum_{k=1}^{n-1-s} \frac{1}{k} \, \, \, mod \, \, n.
\end{split}
\end{equation*}
In particular, when $s=0$ we have, $\sum_{k=1}^{n-1} \frac{1}{k} \equiv 0 \, \, \, mod \, \, n$ if $n>2$.
\end{theorem}
\begin{proof}
We have the following equality (see \cite[identity 1.45]{G}), for any integer $m$:
\begin{equation*}\label{}
\begin{split}
\sum_{k=1}^{m} (-1)^{k-1}\frac{1}{k}\binom{m}{k} = \sum_{k=1}^{m} \frac{1}{k}.
\end{split}
\end{equation*}
We take $m=n-1-s$ and applying \thmref{thm main1}, we derive the main congruence. The second congruence follows from the first one if we set $s=0$. Note that the second congruence was obtained in modulo $p^2$ by J. Wolstenholme in 1862.
\end{proof}
Next, we have another example.
\begin{theorem}\label{thm productsum cong1}
Let $p$ be an odd prime and let $r$ be an integer such that $0 \leq r \leq \frac{p-1}{2}$. Then we have
\begin{equation*}\label{}
\begin{split}
\sum_{j=r}^{\frac{p-1}{2}} \binom{2j}{j}\binom{j}{r} 2^{p-1-2k} \equiv 
\begin{cases} 0 \, \, \, mod \, \, p, & \text{if $r < \frac{p-1}{2}$}\\
(-1)^{\frac{p-1}{2}}\, \, \, mod \, \, p, & \text{if $r = \frac{p-1}{2}$}.
\end{cases}
\end{split}
\end{equation*}
\end{theorem}
\begin{proof}
We use the following equality \cite[pg. 45]{G} for any integers $r$ and $n$:
\begin{equation*}\label{}
\begin{split}
\sum_{j=0}^{\lfloor \frac{n}{2}\rfloor} (-1)^j \binom{n-j}{j}\binom{j}{r}2^{n-2k}=(-1)^r\binom{n+1}{2r+1}.
\end{split}
\end{equation*}
Taking $n=p-1$ for a odd prime $p$ in this equality gives
\begin{equation}\label{eqn sum cong1}
\begin{split}
\sum_{j=0}^{\frac{p-1}{2}} (-1)^j \binom{p-1-j}{j}\binom{j}{r}2^{p-1-2k}=(-1)^r\binom{p}{2r+1}.
\end{split}
\end{equation} 
Since we have $\binom{p-1-j}{j} \equiv (-1)^j\binom{2j}{j} \, \, \, \mod \, \, p$ by \thmref{thm main1}, the result follows from \eqref{eqn sum cong1} by noting that $\binom{p}{2r+1} \equiv 0 \, \, \, \mod \, \, p$ if $r < \frac{p-1}{2}$ and
$\binom{p}{2r+1} \equiv 1 \, \, \, \mod \, \, p$ if $r = \frac{p-1}{2}$.
\end{proof}
The following example shows that one can prove many congruences easily by applying \thmref{thm main1} and \thmref{thm main2}. 
\begin{theorem}\label{thm productsum cong2}
We have the following congruences in modulo an odd prime $p$:
\begin{equation*}\label{}
\begin{split}
F_{p} \equiv \sum_{0 \leq i, \, j \leq p-1} (-1)^{i+j}\binom{i+j}{i}^2 \equiv \sum_{j=0}^{\frac{p-1}{2}} (-1)^j \binom{2j}{j} \equiv 
\begin{cases} 1, & \text{if $p \equiv \pm 1 \, \, \, mod \, \, 5$}\\
-1, & \text{if $p \equiv \pm 2 \, \, \, mod \, \, 5$}.
\end{cases}
\end{split}
\end{equation*}
\begin{equation*}\label{}
\begin{split}
F_{p+1} \equiv \sum_{0 \leq i, \, j \leq p} (-1)^{i+j}\binom{i+j-1}{j}\binom{i+j}{i} \equiv \sum_{j=0}^{\frac{p-1}{2}} (-1)^j \binom{2j-1}{j} \equiv 
\begin{cases} 1, & \text{if $p \equiv \pm 1 \, \, \, mod \, \, 5$}\\
0, & \text{if $p \equiv \pm 2 \, \, \, mod \, \, 5$}.
\end{cases}
\end{split}
\end{equation*}
\end{theorem}
\begin{proof}
We use the following equality \cite[pg. 1570]{KA} for any positive integer $n$:
\begin{equation}\label{eqn sum cong2}
\begin{split}
F_{n+1}=\sum_{0 \leq i, \, j \leq n} (-1)^{i}\binom{n-i}{j}\binom{i+j}{i}.
\end{split}
\end{equation}
Taking $n=p-1$ for an odd prime $p$ and using $\binom{p-1-i}{j} \equiv (-1)^j\binom{i+j}{j}= (-1)^j\binom{i+j}{i}  \, \, \, \mod \, \, p$ by \thmref{thm main1} in \eqref{eqn sum cong2} gives
\begin{equation}\label{eqn sum cong3a}
\begin{split}
F_{p} \equiv \sum_{0 \leq i, \, j \leq p-1} (-1)^{i+j}\binom{i+j}{i}^2 \, \, \, \mod \, \, p.
\end{split}
\end{equation}
Similarly,
taking $n=p$ for an odd prime $p$ and using $\binom{p-i}{j} \equiv (-1)^j\binom{i+j-1}{j} \, \, \, \mod \, \, p$ by \thmref{thm main2} in \eqref{eqn sum cong2} gives
\begin{equation}\label{eqn sum cong3b}
\begin{split}
F_{p+1} \equiv \sum_{0 \leq i, \, j \leq p} (-1)^{i+j}\binom{i+j-1}{j}\binom{i+j}{i} \, \, \, \mod \, \, p.
\end{split}
\end{equation}
On the other hand, we have 
\begin{equation}\label{eqn sum cong4}
\begin{split}
F_{n}=\sum_{j=0}^{\lfloor \frac{n-1}{2} \rfloor} \binom{n-1-j}{j}.
\end{split}
\end{equation}
In particular, for an odd prime $p$ 
\begin{equation}\label{eqn sum cong4a}
\begin{split}
F_{p}=\sum_{j=0}^{\frac{p-1}{2} } \binom{p-1-j}{j}.
\end{split}
\end{equation}
We have $\binom{p-1-j}{j} \equiv (-1)^j \binom{2j}{j} \, \, \, \mod \, \, p$ by \thmref{thm main1}. Thus, the following congruence is obtained from \eqref{eqn sum cong4a}:
\begin{equation}\label{eqn sum cong5a}
\begin{split}
F_{p} \equiv \sum_{j=0}^{\frac{p-1}{2}} (-1)^j \binom{2j}{j} \, \, \, \mod \, \, p.
\end{split}
\end{equation}
Again, for an odd prime $p$ it follows from \eqref{eqn sum cong4} that
\begin{equation}\label{eqn sum cong4b}
\begin{split}
F_{p+1}=\sum_{j=0}^{\frac{p-1}{2} } \binom{p-j}{j}.
\end{split}
\end{equation}
We have $\binom{p-j}{j} \equiv (-1)^j \binom{2j-1}{j} \, \, \, \mod \, \, p$ by \thmref{thm main2}.
Thus, the following congruence is obtained from \eqref{eqn sum cong4b}:
\begin{equation}\label{eqn sum cong5b}
\begin{split}
F_{p+1} \equiv \sum_{j=0}^{\frac{p-1}{2}} (-1)^j \binom{2j-1}{j} \, \, \, \mod \, \, p.
\end{split}
\end{equation}

Finally, we have $F_{p} \equiv 1 \, \, \, mod \, \, p$ and $F_{p+1} \equiv 1 \, \, \, mod \, \, p$ if $p \equiv \pm 1 \, \, \, mod \, \, 5$. We also have $F_{p} \equiv -1 \, \, \, mod \, \, p$ and $F_{p+1} \equiv 0 \, \, \, mod \, \, p$ if $p \equiv \pm 2 \, \, \, mod \, \, 5$ \cite[pg 78-79]{V}.
This completes the proof.
\end{proof}

\textbf{Declaration of competing interest:} The authors declare that they have no known competing financial interests or personal relationships that could have appeared to influence the work reported in this paper.

\end{document}